\documentclass[pdflatex,sn-mathphys-num]{sn-jnl}% Math and Physical Sciences Numbered Reference Style
\usepackage{dirtytalk}
\usepackage{graphicx}%
\usepackage{multirow}%
\usepackage{xcolor}
\usepackage{amsmath,amssymb,amsfonts}%
\usepackage{amsthm}%
\usepackage{mathrsfs,hyperref}%
\usepackage[title]{appendix}%
\usepackage{xcolor}%
\usepackage{textcomp}%
\usepackage{manyfoot}%
\usepackage{booktabs}%
\usepackage{algorithm}%
\usepackage{algorithmicx}%
\usepackage{algpseudocode}%
\usepackage{listings}%
\usepackage{amsmath,amsthm,bbold}
\usepackage{amsfonts,adjustbox,amssymb,xcolor}
\usepackage{rotating,geometry,fullpage,enumitem}
\usepackage{url}
\usepackage[utf8]{inputenc}
\usepackage[english]{babel}
\usepackage{tikz}
\usepackage{pgfplots}
\pgfplotsset{compat=1.18}
\usepackage{hyperref}

\newtheorem{definition}{Definition}
\newtheorem{remark}{Remark}

\theoremstyle{thmstyleone}%
\newtheorem{theorem}{Theorem}%  meant for continuous numbers
\newtheorem{proposition}[theorem]{Proposition}% 
\newtheorem{lemma}{Lemma}
\theoremstyle{thmstyletwo}%
\newtheorem{example}{Example}%

\theoremstyle{thmstylethree}%

\newtheorem{corollary}{Corollary}
\begin{document}

\title[Article Title]{On some F-spaces associated to the functions  satisfying Mulholland inequality.}

\author{\begin{center}Lav Kumar Singh${}^1$\footnote{This research was supported by the Institute of Mathematics, Physics and Mechanics, Ljubljana (Slovenia), within the research program \say{Algebra,  Operator theory and Financial mathematics (code: P1-0222).}} ~and Aljo\v{s}a Peperko${}^{1,2}$\\
	${}^1$Institute of Mathematics, Physics and Mechanics,
	Jadranska 19, SI-1000 Ljubljana, Slovenia.\\${}^2$Faculty of Mechanical Engineering,
	University of Ljubljana, A\v{s}ker\v{c}eva 6, SI-1000 Ljubljana, Slovenia. \\
	Email: lavksingh@hotmail.com; aljosa.peperko@fs.uni-lj.si\end{center}}

%\author*[1]{\fnm{Lav} \sur{Kumar Singh}}\email{lav.singh@imfm.si}

%\author[1,2]{\fnm{Aljoša} \sur{Peperko}}\email{aljosa.peperko@fs.uni-lj.si}

%\affil[]{\begin{center}\orgdiv{$^1$Department of Mathematics}, \orgname{Institute of Mathematics, Physics and Mechanics}, \orgaddress{\street{Jadranska Ulica, 19}, \city{Ljubljana}, \postcode{1000}, \country{Slovenia}}\end{center}}

%\affil[]{\begin{center}\orgdiv{$^2$Faculty of Mechanical engineering}, \orgname{University of Ljubljana}, \orgaddress{\street{Aškerčeva cesta 6}, \city{Ljubljana}, \postcode{1000}, \country{Slovenia}}\end{center}}

\abstract{In this article we explore a new growth condition on Young functions, which we call Mulholland condition, pertaining to the mathematician H.P Mulholland, who studied these functions for the first time, albeit in a different context. We construct a non-trivial Young function $\Omega$ which satisfies Mulholland condition and $\Delta_2$-condition. We then associate  $F$-norms to the vector space $X_1\oplus X_2$, where $X_1$ and $X_2$ are Banach spaces, using the function $\Omega$. This $F$-space contains the Banach space $X_1$ and $X_2$ as a maximal Banach subspace. Further, the Banach envelope $(X_1\oplus X_2,||.||_{\Omega_o})$ of this $F$-space corresponds to the Young function $\Omega_o$ who characteristic function is an asymptotic line to the characteristic function of the Young function $\Omega$. Thus these $F$-spaces serves as "interpolation space" for  Banach spaces $X_1$ and $(X_1\oplus X_2, ||.||_{\Omega_o})$ in some sense. These $F$-space are well behaved in regards to Hahn-Banach extension property, which is lacking in classical $F$-spaces like $L^p$ and $H^p$ for $0<p<1$. Towards the end, some direct sums for Orlicz spaces are discussed.}




\maketitle

\section{Introduction}\label{sec1}
$F$-spaces often arise in various subdomains of operator theory. For an introductory literature on $F$-spaces one can refer the book \say{An F-space sampler}  by N.J Kalton \cite{Kalton}, where he studies non-locally convex topological vector spaces ($F$-spaces in particular). Most of the text is dedicated to the study of $F$-space $L^p[0,1]$ and $H^p$ for $0<p<1$ (the $p$-norm being $||f||_p=\int|f|^pd\mu$). One can construct Banach envelope of any $F$-space, which in some sense is the \say{smallest} Banach space containing the original $F$-space by giving a new norm to the underlying vector space. As mentioned by Kalton, modern functional analysis mostly focuses on locally convex space, and rightly so. But in some cases, there is no reason to restrict the study to locally convex space, for example there is no reason to restrict the study of $H^p$ spaces only for $p\geq 1$. Several important results regarding $F$-spaces are explored in \cite{Cui},\cite{Drewnowski},\cite{Kalton1},\cite{Mastylo} and \cite{Kalton}.\\

 In this article we consider a family of special type of Young functions $\Phi(x)=|x|e^{\chi(\ln |x|)}$, where $\chi$ is a continuous even convex function on $\mathbb R$ which is also strictly increasing on positive axis. The function $\chi$ will be called the characteristic function of $\Phi$. Such Young functions are said to satisfy Mulholland condition. When $\chi$ is an affine function of the type $\chi(x)=ax+b$, then the associated Young function is of type $\Phi(x)=Cx^m$ for some $C>0$ and $m\in \mathbb N$. But these Young functions are of less interest to us. In section \ref{sec3}, we construct a special kind of Young function $\Omega$ whose characteristic function $\chi$ has an ever increasing growth. But the growth of $\Omega$ is still restricted by $\Delta_2$-condition. These characteristic functions have a linear asymptote at infinity, which will serve as characteristic function for Young function $\Omega_o$. Now for any two Banach spaces $X_1$ and $X_2$ and any Young function $\Phi$ with Mulholland condition, we can assign an $F$-norm $||.||_{\Phi}$ to the space $X_1\oplus X_2$. These $F$-norms are actually norms if the characteristic function of $\Phi$ is an affine function. Otherwise these $F$-norms fails homogeneity condition in general. We will show that for our special type of Young function $\Omega$ constructed in example \ref{example}, the $F$-space $(X_1\oplus X_2,||.||_\Omega)$ is not a Banach space and contains both $X_1$ and $X_2$ as maximal Banach subspaces. Further they manifest an interesting geometric phenomenon. The Banach envelope of $(X_1\oplus X_2,||.||_\Omega)$ is $(X_1\oplus X_2,||.||_{\Omega_o})$, where $\Omega_o$ is the Young function whose characteristic function is the asymptotic line to the characteristic function of $\Omega$. Interestingly these spaces also enjoy the Hahn-Banach extension property (HBEP), which is not a common appearance in $F$-spaces. 
\section{Preliminaries to Young functions and associated Orlicz spaces.}
\subsection{Young functions}
\begin{definition} A {\bf Young function} is a convex, left semicontinuous, even  function $\Phi:\mathbb R\to \mathbb [0,\infty]$  such that  $\Phi(0)=0$ and  $\lim_{x\to \infty}\Phi(x)=\infty$.
\end{definition}
A Young function is said to be {\it finite} if $\Phi(x)<\infty$ for each $x\in \mathbb R$. Finite Young functions are automatically continuous because finite convex functions are continuous. Further an Young function $\Phi$ is called {\it N-function} if it is continuous (and hence finite), $\lim_{x\to \infty}\frac{\Phi(x)}{x}=\infty$ and $\lim_{x\to0}\frac{\Phi(x)}{x}=0$. A Young function $\Phi$ is said to satisfy \textit{Mulholland conditions} if it is continuous and strictly increasing on $[0,\infty)$ and $\log \Phi(x)$ is a convex function of $\log x$. In fact it was shown in \cite{Mulholland}[5,Cor. 1] that a strictly increasing (on positive real axis) finite Young function $\Phi$ satisfying Mulholland condition is equivalent to it taking the form $\Phi(x)=|x|e^{\chi(\log |x|)}$ for some continuous increasing convex function $\chi$. In such situation, we say that $\chi$ is the characteristic function of $\Phi$.\\

\begin{example} The functions such as $\sinh |x|$, $|x|^{1+a}e^{b|x|^c}$ ($a,b,c>0$) are strict Young functions which satisfies Mulholland conditions.\end{example}

A Young function $\Phi$ is said to satisfy \textit{$\Delta_2$-condition} if there exists $M>0$ such that $\Phi(2x)\leq M\Phi(x)$ for all $x\geq x_0\geq 0$. For example, if we take $\chi$ to be an affine function on $\mathbb R$ with positive slope, then 
$\Phi(x)=|x|e^{\chi(\log |x|)}$ gives us a trivial Young function which is strict and satisfies both Mulholland condition and $\Delta_2$-conditions. In-fact there is an ubandance of strict Young functions which satisfy both these conditions. \\

Associated to a Young function $\Phi$, there exists another convex function $\Psi$ given by $$\Psi(y)=\sup\{x|y|-\Phi(x):~x\geq 0\},~y\in \mathbb R.$$ The function $\Psi$ is also a Young function and the pair $(\Phi,\Psi)$ is called complementary pair of Young functions. The complementary pair $(\Phi,\Psi)$ of Young functions is said to satisfy $\Delta_2$-condition if both $\Phi$ and $\Psi$ satisfy $\Delta_2$-condition.\\

\begin{remark}
	For $p\geq 1$, the maps $t\mapsto |t|^p$ are Young functions which satisfy both the $\Delta_2$-condition and Mulholland's condition. Further the associated complementary Young function  also satisfies $\Delta_2$. In-fact taking $\chi$ to be any affine function of positive slope on $\mathbb R$ makes $\Phi(t)=|t|e^{\chi(\log |t|)}$ a Young function which satisfy Mulholland condition and the complementary pair associated to it satisfies $\Delta_2$-condition.
\end{remark}
\begin{example} The Young function $\Phi:\mathbb R\to \mathbb R^+$ defined as $t\mapsto |t|e^{\frac{t^2-|t|}{2|t|}}$ is strict and satisfy Mulholland's condition. Further the complementary pair $(\Phi,\Psi)$ satisfies $\Delta_2$-condition. Clearly, $\Phi(t)=|t|e^{\sinh (\log|t|)}$ and hence it satisfies  Mulholland condition since $\sinh$ is strictly increasing convex function on $[0,\infty)$. Now notice that \begin{eqnarray*}
	\frac{\Phi(2t)}{\Phi(t)}&=2e^{\frac{2t^2+1}{4t}}\rightarrow \infty
\end{eqnarray*}
And hence, $\Phi$ does not satisfy $\Delta_2$-condition.
\end{example}
\subsection{Orlicz spaces}\label{Orlicz}
We define Orlicz space associated to a measure space $(X,\mu)$ and Young function $\Phi$ as $$L^\Phi(X)=\{f:X\to\mathbb C~:~f~\text{is measurable}~,\int_X\Phi(\beta |f|)d\mu<\infty~\text{for some}~\beta>0\}$$
$L^\Phi(X)$ becomes a Banach space with respect to the  Gauge norm $$N_\Phi(f)=\inf\left\{k>0~:~\int_X\Phi\left(\frac{|f|}{k}\right)d\mu\leq 1\right\}$$
There is an another equivalent norm on $L^\Phi(X)$, known as Orlicz norm $$\|f\|_\Phi=\sup\left\{\int_X|fg|d\mu~:~g\in L^\Psi(X)~\text{and}~\int_X\Psi(|g|)d\mu\leq 1\right\}.$$
Further, if $(\Phi,\Psi)$ are a pair of complementary Young functions, both satisfying $\Delta_2$-condition, then $\left(L^\Phi(X), N_\Phi\right)$ is a reflexive Banach space with  $\left(L^\Psi(X),\|.\|_\Psi\right)$ as its dual space.
\subsection{F-spaces}
\begin{definition}
	An $F$-space is a vector space $X$ over a field of real or complex numbers together with a metric $d:X\times X\to \mathbb [0,\infty)$ such that \begin{itemize}
		\item Scalar multiplication in $X$ is continuous with respect to the metric $d$ on $X$ and standard metric on $\mathbb C$ (or $\mathbb R$).
		\item Addition in $X$ is continuous with respect to $d$.
		\item $d$ is translation invariant i.e, $d(x+a,y+a)=d(x,y)$ for all $x,y,a\in X$.
		\item  Metric space $(X,d)$ is complete.
	\end{itemize}
	The map $x\mapsto \|x\|_F=d(x,0)$ is called \emph{F-norm}. An $F$-space is called \emph{Fréchet space} if the underlying topology is locally convex.
\end{definition}
\begin{definition}
	A topological vector space $X$ is called locally bounded if there exists an open neighborhood $A$ of $0$ such that for each open neighborhood $U$ of $0$ there exists  $t>0$ such that $A\subset  sU$ for all $s>t$.
\end{definition}
$F$-spaces are not locally convex in general and the closed unit ball $B_X$ is not convex. Hence, it is natural to consider the closed convex hull $co(B_X)$. The Gauge seminorm on X
$$\|x\|_C=\inf\left\{\lambda>0~:~\frac{f}{\lambda}\in co(B_X)\right\}$$
is actually a norm on $X$ if its dual $X^*$ separate points (see \cite[Ch.2, sec. 4]{Kalton}) and the identity mapping $i:(X,\|.\|_F)\to (X,\|.\|_C)$ is continuous. In-fact $\overline{(X,\|.\|_C)}$ is the \say{smallest} Banach space containing $(X,\|.\|_F)$.
\begin{definition}
	Let $(X,d)$ be an F-space with a separating dual. The Banach space $\overline{(X,\|.\|_C)}$ generated by the Gauge norm on $X$ is called the Banach envelope of the $F$-space $(X,d)$.
\end{definition}
\begin{example}
	Consider the $F$-space $(\mathbb R^2,\|.\|_p)$ for any $0<p<1$. Then the closed convex hull of the unit ball of $\mathbb R^2$ with respect to $\|.\|_p$ is same as the unit ball of $(\mathbb R^2,\|.\|_1)$ and hence $(\mathbb R^2,\|.\|_1)$ is the Banach envelope of $(\mathbb R^2,\|.\|_p)$.
\end{example}
\begin{theorem}\cite[Ch.2 Sec. 4]{Kalton}
	If $(X,d)$ is an $F$-space then the dual space (space of continuous linear functionals on $X$ with respect to $F$-norm) of $(X,d)$ is the same (isometrically isomorphic) as the dual space of its Banach envelope.
\end{theorem}

	\section{Interpolation $F$-spaces of Banach spaces}\label{sec3}
We notice that there is an ubundance of non-trivial Young functions which satisfy both $\Delta_2$-condition and Mulholland condition.
	In the previous remark, we saw few trivial examples of Young functions which satisfy Mulholland condition and the complementary pair $(\Phi,\Psi)$  satisfies $\Delta_2$-condition. We now construct a non-trivial example of such a Young function $\Phi$. 
	\begin{example}\label{example} Let $M>0$ be a fixed number and $\{m_i\}_{i=0}^\infty$ be a strictly increasing sequence of positive real numbers converging to a fixed $M>0$. Consider the convex continuous even function $\chi:\mathbb R\to \mathbb R^+$ defined as piecewise straight lines on the intervals $[0,\ln 2],[\ln2,2\ln 2],\cdots,[r\ln 2,(r+1)\ln2],...$.The slope of the line in interval $[r\ln 2,(r+1)\ln2]$ is $m_{r}$. The growth of $\chi$ is depicted in figure below.\\
		
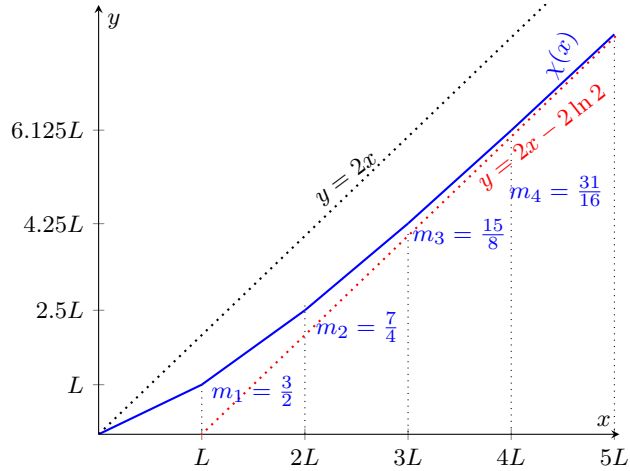
		\begin{center}
		
			\begin{tikzpicture}
			\begin{axis}[
				axis lines=middle,
				xlabel={$x$},
				ylabel={$y$},
				xmin=0, xmax=3.5,
				ymin=0, ymax=6,
				xtick={0,0.6931,1.3862,2.0794,2.7726,3.4657},
				xticklabels={$0$,$L$,$2L$,$3L$,$4L$,$5L$},
				ytick={0,0.6931,1.7327,2.9430,4.2494,6.1225,8.06},
				yticklabels={$0$,$L$,$2.5L$,$4.25L$,$6.125L$,$8.0625L$}
				]
				
				% Piecewise function f(x)
				\addplot[blue, thick, domain=0:0.6931] {x};
				\addplot[blue, thick, domain=0.6931:1.3862] {0.6931 + 1.5*(x - 0.6931)};
				\addplot[blue, thick, domain=1.3862:2.0794] {2.5*0.6931 + 1.75*(x - 2*0.6931)};
				\addplot[blue, thick, domain=2.0794:2.7726] {4.25*0.6931 + 1.875*(x - 3*0.6931)};
				\addplot[blue, thick, domain=2.7726:3.4657] {6.125*0.6931 + 1.9375*(x - 4*0.6931)};
					% Red dotted line y = 2x - 2ln2
				\addplot[red, dotted, thick, domain=0:3.5] {2*x - 2*ln(2)};
				\node[rotate=38, anchor=south west, red] at (axis cs:2.6,3.2) {$y=2x - 2\ln 2$};
				% Adjusted label for f(x), closer to the line
				\node[rotate=63, anchor=west, blue] at (axis cs:3,4.9) {$\chi(x)$};
				\node[rotate=0, anchor=west, blue] at (axis cs:0.7,0.6) {$m_1=\frac{3}{2}$};
				\node[rotate=0, anchor=west, blue] at (axis cs:1.4,1.5) {$m_2=\frac{7}{4}$};
				\node[rotate=0, anchor=west, blue] at (axis cs:2.05,2.8) {$m_3=\frac{15}{8}$};
				\node[rotate=0, anchor=west, blue] at (axis cs:2.7,3.4) {$m_4=\frac{31}{16}$};
				\node[rotate=0, anchor=west, blue] at (axis cs:1,-0.5) {Figure: Plot of $\chi(x)$};
				
				% Dotted line y = 2x
				\addplot[dotted, thick, domain=0:3.5] {2*x};
				
				% Label for y = 2x
				\node[rotate=37, anchor=south west] at (axis cs:1.5,3) {$y=2x$};
				
				% Dotted perpendiculars from f(x) to x-axis at x = r*ln2
				\draw[dotted] (axis cs:0.6931,0.6931) -- (axis cs:0.6931,0);
				\draw[dotted] (axis cs:1.3862,1.8) -- (axis cs:1.3862,0);
				\draw[dotted] (axis cs:2.0794,2.8) -- (axis cs:2.0794,0);
				\draw[dotted] (axis cs:2.7726,4.1) -- (axis cs:2.7726,0);
				\draw[dotted] (axis cs:3.4657,5.5) -- (axis cs:3.4657,0);
				
			\end{axis}
		\end{tikzpicture}
	\end{center}

		 Now if we define $\Omega(t)=|t|e^{\chi(\log|t|)}$, then $\Omega$ is a Young function which satisfies Mulholland condition. It is very desirable for Young functions and their complementary functions to have $\Delta_2$-condition because it makes the associated Orlicz spaces reflexive and the simple functions becomes dense in $L^\Phi(X)$. Thus we show that the above constructed Young function is actually well behaved with respect to the growth and possess $\Delta_2$-condition.  Notice that for $t>0$ such that $\log t\in [r\log 2,(r+1)\log 2]$, we have
		\begin{align*}
			\frac{\Omega (2t)}{\Omega(t)}&=2e^{\chi(\log 2+\log t)-\chi(\log t)}\\
			&\leq2e^{m_{r+1}\ln 2}\\
			&\leq 2e^{M\ln 2}=2^{M+1}.
		\end{align*}
		
		Hence, $\Omega$ satisfies $\Delta_2$-condition. Although it is not necessary for our further results in this section, we will see that the complementary function $\Theta$ also satisfies $\Delta_2$-condition. To ease the computation, we fix $m_r=2-\frac{1}{2^r}$. Let $\theta(x,t)=xt-\Omega(x)$ be a function defined for $x,t\geq 0$. On interval $I_r=[2^r,2^{r+1}]$ the function $\Omega(x)$ takes the form $\Omega(x)=C_rx^{m_r+1}$, where $C_r$ is a constant. Let $t$ be fixed. The derivative of $\theta$ on the interval $I_r$ is  
		\begin{align*}
			\frac{d}{dx}\theta(x,t)=t-C_r(m_{r}+1)x^{m_r}.
		\end{align*}
		Thus, $\theta(x,t)$ attains a maximum value in interval $I_r$ at point $x_r(t)=\left(\frac{t}{C_r(m_r+1)}\right)^{\frac{1}{m_r}}$ if $x_r(t)\in I_r$ or on the end points of the interval. The maximum possible value of $\theta(x,t)$ on interval $I_r$ is \begin{eqnarray}\label{max}\Theta_r(t)=tx_r(t)-C_r(x_r(t))^{1+m_r}=\frac{m_r}{(1+m_r)^{\frac{1+m_r}{m_r}}}C_r^{-\frac{1}{m_r}}t^{\frac{1+m_r}{m_r}}\end{eqnarray}
		Now $\Theta(t)=\sup_{r\geq 0}\Theta_r(t)$. For a fixed $t$, the supremum is attained at some $r=r(t)$, i.e, $\Theta(t)=\Theta_{r(t)}(t)$. Now, consider the ratio $R(t)=\frac{\Theta(2t)}{\Theta(t)}=\frac{\Theta_{r(2t)}(2t)}{\Theta_{r(t)}(t)}$. Since, $\Theta_r(t)$ grows like $t^{\frac{m_r+1}{m_r}}$, the maximizing $r(t)$ increase with an increase in $t$. Hence,$$R(t)\leq \max\{\frac{\Theta_r(2t)}{\Theta_r(t)}: r\in \mathbb N\}\leq 2^{1+\frac{1}{m_r}} ~~~~~~~~~~~~~~~~~~~~~~(\text{due to} ~\ref{max})$$
		As $m_r\to 2$ for large values of $r$, we see that $R(t)\leq 2^{3/2}$ eventually. Thus the complementary function $\Theta$ satisfies $\Delta_2$-condition.\\
		
		\noindent {\bf Note:} We will also need the Mulholland Young function $\Omega_o$ associated to the asymptotic red line $y=2x-2\ln 2$ in the above figure. In-fact $\Omega_o(x)=|x|e^{2\ln |x|-2\ln2}=\frac{1}{4}|x|^3$
	\end{example}
	\begin{example}
		Let $\chi$ denote a function which is a combination of rotation of the graph of $f(x)=e^{-x}$ about origin by angle $\theta$ and appropriate translation such that the resulting curve has the  $y=x\tan\theta -c$ as an asymptote. Then $\Omega(x)=|x|e^{\chi(\ln |x|)}$ is an Young function which satisfies Mulholland condition and $\Delta_2$-condition.
	\end{example}
	Suppose $(X_i,\|.\|_i)$ are $F$- spaces for $i=1,2,..,n$ and $\Phi$ is a Young function which satisfies Mulholland condition. We consider the vector space  direct sum $\oplus_{i=1}^nX_i$. Further we define a metric $d_\Phi:\oplus_{i=1}^nX_i\times\oplus_{i=1}^nX_i\to [0,\infty)$ as $$d_\Phi\left((x_1,...,x_n),(y_1,...,y_n)\right)=\Phi^{-1}\left(\sum_{i=1}^n\Phi\left(\|x_i-y_i\|_i\right)\right).$$
	The associated $F$-norm is $\|(x_1,...,x_n)\|_F=\Phi^{-1}\left(\sum_{i=1}^n\Phi(\|x_i\|_i)\right)$.
	One can easily check that this is a well defined metric. $d_\Phi\left((x_1,...,x_n),(y_1,...,y_n)\right)=0$ if and only $x_i=y_i$ for each $i=1,...,n$. Due to Mulholland's condition on $\Phi$, the following Minkowsky type inequality (also known as \emph{Mulholland's inequality}) holds true (see \cite{Mulholland}[5, Th. 1]). 
	
	$$\Phi^{-1}\left(\sum_{i=1}^n \Phi(a_i+b_i)\right)\leq \Phi^{-1}\left(\sum_{i=1}^n \Phi(a_i)\right)+\Phi^{-1}\left(\sum_{i=1}^n \Phi(b_i)\right)$$ for all $a_i,b_i\geq 0$.
	And hence the triangle inequality follows for the $F$-norm on $\oplus_{i=1}^nX_i$ and hence for the metric $d_\Phi$. We denote this metric space by $\left(\oplus_{i=1}^n  X_i, d_\Phi\right)$.
	\begin{proposition}
		Metric space $\left(\oplus_{i=1}^nX_i, d_\Phi\right)$ is an $F$-space if $\Phi$ satisfies Mulholland condition.
	\end{proposition}
	\begin{proof}
		We start by showing that the scalar multiplication continuous with respect to the metric $d_\Phi$. Suppose $\left\{K_\alpha,(x^{(\alpha)}_1,...,x^{(\alpha)}_n)\right\}_{\alpha\in \Gamma}$ be a net in $\mathbb C\times \oplus_{i=1}^nX_i$ converging to $\left\{K,(x_1,...,x_n)\right\}$. Then $||K_\alpha x_i^{(\alpha)}-Kx_i||_i\to 0$ for each $i$ due the continuity of scalar multiplications in each space $(X_i,\|.\|_i)$.
	Combining this with the fact that $\Phi$ is strictly increasing and continuous tells us that $$d_\Phi\left((K_\alpha x_1^{(\alpha)},...,K_\alpha x^{(\alpha)}_n),(Kx_1,...,Kx_n)\right)=\Phi^{-1}\left(\sum_{i=1}^n\Phi(\|K_\alpha x^{(\alpha)}_i-Kx_i\|_i)\right)\to 0.$$
		Thus the scalar multiplication is continuous with respect to the metric $d_\Phi$.\\
		Now we turn to prove that the addition in $\oplus_{i=1}^nX_i$ is continuous with respect to the metric $d_\Phi$. But this is evident from the fact that addition is a short map under the metric $d_\Phi$ and hence continuous. The fact that $d_\Phi$ is translation invariant follows from the fact that the $F$-norm, which is translation invariant, generates $d_\Phi$. Only thing remains to be verified is whether $\left(\oplus_{i=1}^n X_i,d_\Phi\right)$  is a complete metric space. To see this, let $\{(x_1^{(r)},...,x_n^{(r)})\}_{r=1}^\infty$ be a Cauchy sequence in $\left(\oplus_{i=1}^nX_i, d_\Phi\right)$. Then for each $\epsilon>0$, there exist a $N_\epsilon$ such that $d_\Phi((x_i^{(r)}),(x_i^{(s)}))<\epsilon$ for all $r,s>N_\epsilon$. Then
		\begin{align}
			d_\Phi\left((x_i^{(r)}),(x_i^{(s)})\right)&=\Phi^{-1}\left(\sum_{i=1}^n\Phi (\|x_i^{(r)}-x_i^{(s)}\|_i)\right)<\epsilon\nonumber
			\end{align}
			for all $r,s>N_{\epsilon}$. Hence, due to $\Phi$ being strictly increasing, we have $\sum_{i=1}^n\Phi(\|(x_i^{(r)})-x_i^{(s)}\|_i)<\Phi(\epsilon)$ for all $r,s>N_{\epsilon}$. Hence for each $i$, again due to strictly increasing nature of $\Phi$, we have $\|x_i^{(r)}-x_i^{(s)}\|_i<\Phi^{-1}\Phi(\epsilon)=\epsilon$ for all $r,s>N_\epsilon$. Thus, for each $i$, the sequence $\{x_i^{(r)}\}_{r=1}^\infty$ is a Cauchy sequence in the $F$-space $X_i$. Thus there exists $x_i\in X_i$ for each $i=1,2,..,n$ such that $\{x_i^{(r)}\}\overset{\|.\|_i}{\to} x_i$ for each $i=1,2,..,n$. Now, we claim that $\{(x_1^{(r)},...,x_n^{(r)})\}\overset{d_\Phi}{\to}(x_1,...,x_n)$. But this is the easy consequence of the continuity of $\Phi$ and $\Phi^{-1}$. Hence, $\left(\oplus_{i=1}^nX_i,d_\Phi\right)$ is an $F$-space.
	\end{proof}
	
	As one might have noticed, for Banach spaces $X_1,..,X_n$, the only thing preventing $\left(\oplus_{i=1}^nX_i,d_\Phi\right)$ from being a Banach space is the homogeneity of $F$-norm $x\mapsto d_\Phi(0,x)$ with respect to scalar multiplication. In-fact the $F$-norm is homogeneous if $\Phi(x)=|x|e^{\chi(\ln |x|)}$, where $\chi(x)=ax+b$ is any affine function. On the contrary if $\Phi$ is of the type constructed in example \ref{example}, then it is easy to see that the $F$-norm on $\oplus_{i=1}^nX_i$ is not homogeneous with respect to scalar multiplication (the ever changing growth of $\Phi$ would not permit homogeneity, and it also follows from the next result).	Consider the set $\mathcal W=\left\{V~:~V\prec \oplus_{i=1}^nX_i,~ \|\alpha(x_i)\|_F=|\alpha|.\|(x_i)\|_F~\forall \alpha\in \mathbb C, (x_i)\in \oplus_{i=1}^nX_i  \right\}$. Define the natural ordering on $\mathcal W$ induced by inclusion i.e $V_1\prec V_2$ if $V_1$ is a subspace of $V_2$. Then $\left(\mathcal W, \prec\right)$ becomes a partially ordered set. Further, $\mathcal W$ is non-empty because $V=\{(x,0,..,0)~:~ x\in X_1\}\in \mathcal W$. If $\{W_\alpha\}$ is a  chain in $\mathcal W$, then $\cup W_\alpha$ is its upper bound. Hence, by Zorn's lemma the family $\mathcal W$ has at-least one maximal member. Clearly, the maximal member of $\mathcal W$  will be closed with respect to the metric $d_\Phi$. Hence, the maximal member of $\mathcal W$ will be a Banach space with respect to the norm $\|.\|_F$. We will see that for $\Phi$ constructed in example \ref{example}, each space $X_i$ can be identified isometrically (w.r.t $d_\Phi$) with a maximal member of $\mathcal W$.
	\begin{definition}
		A locally bounded $F$-space $(X,d)$ is called a $p$-interpolation $F$-space for Banach spaces $X_1$, $X_2$ and a $p\in [1,\infty)$ if
		\begin{enumerate} 
			\item there exists a distance preserving linear maps $i_1:X_1\to X$ and $i_2:X_2\to X$ such that $i_1(X_1)$ and $i_2(X_2)$ are maximal Banach subspaces in $X$.
			\item the Banach envelope of $(X,d)$ is the $p$-direct sum  $X_1\oplus_p X_2$.
		\end{enumerate}
	\end{definition}
	\begin{theorem}
		If $\Omega$ is the Young function from example \ref{example} and $(X_i,\|.\|_i)$ be Banach spaces for $i=1,2,...,n$. Then each Banach space  $\left(X_i,\|.\|_i\right)$ is isometrically isomorphic to a maximal Banach subspace of the $F$-space $\left(\oplus_{i=1}^nX_i,d_\Omega\right)$.
	\end{theorem}
	\begin{proof}
		To simplify the computations, we shall prove it for $n=2$. The general case follows in similar fashion. Consider the natural embedding $\theta:(X_1, \|.\|_1)\to \left(X_1\oplus X_2, d_\Omega\right)$ given by $\theta(x)=(x,0)$. This is clearly an isometry (w.r.t metric $d_\Omega$).  We claim that $\theta(X_1)$ is a maximal Banach subspace of the $F$-space $\left(X_1\oplus X_2,d_\Omega\right)$. To prove this, suppose  on the contrary, that $\theta(X_1)$ is not maximal. Then there exist $W\in \mathcal W$ such that $\theta(X_1)\prec W$. Hence, there exists a $(x,y)\in W\setminus \theta(X_1)$ such that \begin{equation}\label{homogeneous}\Omega^{-1}\left(\Omega(\|\alpha x\|_1)+\Omega(\|\alpha y\|_2)\right)=|\alpha|\Omega^{-1}\left(\Omega(\|x\|_1)+\Phi(\|y\|_2)\right)~~~~\forall \alpha\in \mathbb C.\end{equation}  Without loss of generality, we can choose $(f,g)$ to be such that $u=\|x\|_1= 1$ and $v= \|y\|_2= 1$. Thus due to equation \ref{homogeneous}, we have \begin{equation}\label{homogeneouso}
			\frac{\Omega\left(2|\alpha|\Omega^{-1}(\Omega(u)+\Omega(v))\right)}{	\Omega(2|\alpha| u)+\Omega(2|\alpha| v)}=1~~~~\forall \alpha\in \mathbb C
		\end{equation}
		Hence, 
		\begin{align}\label{homegeneous1}
	\frac{\Omega\left(2|\alpha|\Omega^{-1}(\Omega(u)+\Omega(v))\right)}{	\Omega(2|\alpha| u)+\Omega(2|\alpha| v)}&=	\frac{\Omega\left(2|\alpha|\Omega^{-1}(\Omega(u)+\Omega(v))\right)}{|\alpha|\Omega^{-1}(\Omega(u)+\Omega(v))}\frac{|\alpha|\Omega^{-1}(\Omega(u)+\Omega(v))}{	\Omega(2|\alpha| u)+\Omega(2|\alpha| v)}\nonumber\\
	&\leq M\frac{|\alpha|}{\Omega(2|\alpha| u)+\Omega(2|\alpha| v)}~~~~~~~\forall \alpha\in \mathbb C~\text{large enough}&& (\because \Omega~\text{is}~\Delta_2)
	\end{align}
	Now notice that for large values of $t$ the function $\Omega(t)\approx \frac{1}{4}t^3$. Hence, $\Omega(2|\alpha|u)\approx 2|\alpha^3|u^3$ and $\Omega(2|\alpha|v)\approx 2|\alpha^3|v^3$. For large values of $\alpha$.
	Thus from the equation \ref{homegeneous1}, we have $$\lim_{\alpha\to \infty}\frac{\Omega\left(2|\alpha|\Omega^{-1}(\Omega(u)+\Omega(v))\right)}{	\Omega(2|\alpha| u)+\Omega(2|\alpha| v)}\leq \lim_{\alpha\to \infty}M\frac{|\alpha|}{2|\alpha|^3(u^3+v^3)}=0,$$
	which is a clear contradiction to the equation \ref{homogeneouso}.\\
	
	Hence, $\theta(X_1)$ is a maximal Banach subspace of the $F$-space $(X_1\oplus X_2, d_\Omega)$.
	\end{proof}
	\begin{remark}
		The above phenomenon is peculiar for the special Young function $\Omega$ constructed in example \ref{example}. For Young functions of the type $\Phi(x)=cx^m$, the subspace $X_i$ is not a maximal Banach subspace in $(X_1\oplus X_2,d_\Phi)$ because the latter is a Banach space itself in such case.
	\end{remark}
	
	In the preceding theorem, we have interpolated the $F$-space $(X_1\oplus X_2,d_\Omega)$ from inside through maximal Banach space. Now in the next result we interpolate it from the outside i.e we compute its Banach envelope. Recall from example \ref{example}, the Young function $\Omega_o$ whose characteristic function was an asymptote to the characteristic function of $\Omega$. We now demonstrate an interesting phenomena which outlines that a Banach envelope of an $F$-space associated to a Young function $\Omega$ with Mulholland condition is nothing but the space associated to the Young function whose characteristic function is an asymptote to the characteristic function of $\Omega$. Since, $(X_1\oplus X_2,d_\Omega)$ has a separating dual and is locally bounded (see theorem \ref{locb}), we have the following result.
	\begin{theorem}
		If $X_1$ and $X_2$ are Banach spaces, then the Banach envelope of the $F$-space $(X_1\oplus X_2,d_\Omega)$ is the Banach space $(X_1\oplus X_2,d_{\Omega_o})$.
	\end{theorem}
	\begin{proof}
		Recall that $\|(x_1,x_2)\|_{\Omega_o}=d_{\Omega_o}\left((0,0),(x_1,x_2)\right)=\left(\|x_1\|^3+\|x_2\|^3\right)^{1/3}$. We just have to prove that the gauge norm is $p(x_1,x_2)=\inf\left\{\lambda>0:~\frac{(x_1,x_2)}{\lambda}\in co(B)\right\}=\left(\|x_1\|^3+\|x_2\|^3\right)^{1/3}$, where $co(B)$ denotes the closed convex hull of the unit ball of $(X_1\oplus X_2, d_\Omega)$. Let $k>0$ be very large. Further let $\lambda=\left(\|kx_1\|^3+\|kx_2\|^3\right)^{1/3}$. Then
		\begin{align*}
			\left\vert\left\vert\frac{(kx_1,ky_2)}{\lambda}\right\vert\right\vert_\Omega&=\Omega^{-1}\left(\Omega\left(\frac{\|kx_1\|}{\lambda}\right)+\Omega\left(\frac{||kx_2||}{\lambda}\right)\right)\\&\leq \Omega_o^{-1}\left(\Omega\left(\frac{\|kx_1\|}{\lambda}\right)+\Omega\left(\frac{\|kx_2\|}{\lambda}\right)\right)&&\because \Omega^{-1}\leq \Omega_o^{-1}~\text{eventually}\\&\leq \Omega_o^{-1}\left(\frac{1}{\lambda}\Omega(\|kx_1\|)+\frac{1}{\lambda}\Omega(\|kx_2\|)\right)&&\because \Omega~\text{is convex and}~\lambda~\text{is very large}\\&\approx\Omega_o^{-1}\left(\frac{1}{4\lambda}\|kx_1\|^3+\frac{1}{4\lambda}\|kx_2\|^3\right)&&\because \Omega(t)=\frac{1}{4}t^3~\text{for large}~t\\\therefore\left\vert\left\vert\frac{(kx_1,ky_2)}{\lambda}\right\vert\right\vert_\Omega&=\frac{1}{\lambda}\left(\|kx_1\|^3+\|kx_2\|^3\right)^{1/3}+\epsilon_k\\&=1+\epsilon_k
		\end{align*}
		where, $\epsilon_k\to 0$ as $k\to \infty$. Thus, $p(kx_1,kx_2)\leq (1+\epsilon_k)\lambda$ for large $k$. But $p$ is a norm and hence $p(kx_1,kx_2)=|k|p(x_1,x_2)$ for all $k$. Hence, $p(x_1,x_2)\leq \left(\|x_1\|^3+\|x_2\|^3\right)^{1/3}=\|(x_1,x_2)\|_{\Omega_o}$. To prove the reverse inequality, fix $a\in X_1$ and $b\in X_2$. Then, define a  function $p_o:\mathbb R^2\to [0,\infty)$ as $p_o(s,t)=p(sa,tb)$. This is a well defined function and it can be easily verified that this is convex and homogeneous with respect to scalar multiplication on $\mathbb R^2$. If $p$ is a norm other than $\|.\|_{\Omega_o}$, then the associated convex functions $p_o$ should be sandwiched between  $p_{\Omega_o}$ and $p_\Omega$, where $p_{\Omega_o}(s,t)=\Omega_o(\sqrt{s^2+t^2})$ and $p_\Omega(s,t)=\Omega(\sqrt{s^2+t^2})$. But the only homogeneous convex functions sandwiched between $\Omega_o$ and $\Omega$ is $\Omega_o$ itself (since its characteristic function is asymptote to $\Omega$). Hence, $(X_1\oplus X_2,d_{\Omega_o})$ is the Banach envelope of the $F$-space $\left(X_1\oplus X_2, d_\Omega\right)$.
	\end{proof}
	\noindent{\bf Summary:} We have so far worked with $\Omega$, where the sequence of slopes for characteristic function is assumed to be $m_r=2-\frac{1}{2^r}$ for the ease of computation. But all of the above results holds in a general case as well. The following result captures the essence of if in general setup.
	\begin{theorem}
Let $X_1$ and $X_2$ be two Banach spaces and $\left(X_1\oplus X_2,\|.\|_p\right)$ be their $p$-direct sum for some $1\leq p<\infty$. Then there exists a $p$-interpolation $F$-space $\left(X_1\oplus X_2,\|.\|_\Omega\right)$ which contains $X_1$ and $X_2$ as maximal Banach subspaces and the Banach envelope of $\left(X_1\oplus X_2,\|.\|_\Omega\right)$ is the $p$-direct sum  $\left(X_1\oplus X_2,\|.\|_p\right)$.
	\end{theorem}
	\begin{proof}
		Consider the Young function $\Omega_p=\frac{1}{p}|x|^p$. Then its characteristic function is $\chi_p(x)=(p-1)x-\ln p$. Now choose a increasing sequence $\{m_r\}_{r=1}^\infty$ of positive real numbers such that $m_r\to p-1$. Construct a continuous even function $\chi:\mathbb R\to \mathbb R$ of piece-wise straight lines of slope $m_r$ and such that $\chi$ has $\chi_p$ as an asymptote (this is always possible). Then the $F$-space $\left(X_1\oplus X_2,d_{\Omega}\right)$ associated to the Young function $\Omega(x)=|x|e^{\chi(\ln |x|)}$ is the required interpolation space.
	\end{proof}
	\begin{corollary}
		Let $\Omega$ be an Young function with Mulholland condition such that the line $\chi_p(x)=(p-1)x-\ln p$ is asymptote to $\Omega$ (as constructed in preceding theorem for $1\leq p<\infty$). Then the dual of the $F$-space $\left(X_1\oplus X_2,d_\Omega\right)$ is the Banach space $\left(X_1^*\oplus X_2^*, \|.\|_q\right)$, where $\frac{1}{p}+\frac{1}{q}=1$..
	\end{corollary}
	\begin{proof}
		Follows easily from the preceding theorem and the fact that the continuous dual of an $F$-space is equal to the continuous dual of it  Banach envelope.
	\end{proof}
	\begin{remark}
		One might be tempted to ask- what about the $F$-space $X_1\oplus X_2$ equipped with $\|(x,y)\|_p=\|x\|^p+\|y\|^p$ for some $p< 1$? Well, these are non-locally convex $F$-space and they do not carry a copy of $X_1$ or $X_2$, let alone contain them as maximal subspaces. Also, its Banach envelope is always $\left(X_1\oplus X_2,\|.\|_1\right)$. Further they do not possess Hahn-Banach extension property and hence deemed as object of less interest. 
	\end{remark}
	\begin{definition}
		An $F$-space $(X,d)$ is said to have the Hahn-Banach Extension property (HBEP) if for any closed subspace $M$ of $X$ and any continuous linear functional $\varphi:M\to\mathbb C$ has a continuous extension $\varphi':X\to \mathbb C$ such that $\varphi'(x)=\varphi(x)$ for all $x\in M$.
	\end{definition}
	 It is proved in \cite[Ch. 2,3]{Kalton} that the spaces $\ell^p$ and $H^p$ does not have HBEP for $0<p<1$. Motivated by these, Duren, Romberg and Shields formulated the problem in 1969-\say{{\bf Is every $F$-space $X$ with HBEP locally convex?}}. Shapiro answered this question in affirmative if $X$ has a basis \cite{Shapiro} . The answer to this question was proved to be affirmative in general by N.J Kalton (see \cite[Theorem 4.8]{Kalton}).
	 For two Banach spaces $X_1$, $X_2$ and the Young function $\Omega$ from example \ref{example}, the interpolation $F$-space $\left(X_1\oplus_\Omega X_2, d_\Omega\right)$ turns out to be locally convex and hence a Fréchet space, since it posses HBEP, as proved in the next result.
	 \begin{theorem}
	 	Let $X_1$, $X_2$ be two Banach spaces and $\Phi$ be a Young function with Mulholland condition. Then, the $F$-space $\left(X_1\oplus X_2,d_\Phi\right)$ has the HBEP.
	 \end{theorem}
	 \begin{proof}
	 	Let $M$ be a closed subspace of $(X_1\oplus_\Phi X_2, d_\Phi)$ and $\varphi:M\to \mathbb C$ be a continuous linear functional. Consider the projection subspace $\pi_1(M)$ and $\pi_2(M)$ in $X_1$ and $X_2$ respectively. Let $\varphi_{\pi_1}:\pi_1(M)\to \mathbb C$ and $\varphi_{\pi_2}:\pi_2(M)\to \mathbb C$ be the restriction maps i.e $\varphi_{\pi_1}(x)=\varphi(x,0)$ and $\varphi_{\pi_2}(y)=\varphi(0,y)$. Then $\varphi_{\pi_1}$ and $\varphi_{\pi_2}$ are continuous linear maps. By Hahn-Banach theorem, there exists a continuous extensions $\varphi_1:X_1\to \mathbb C$ and $\varphi_2:X_2\to \mathbb C$ of $\varphi_{\pi_1}$ and $\varphi_{\pi_2}$ respectively. Let $\varphi':X_1\oplus X_2\to \mathbb C$ be defined as $\varphi'(x,y)=\varphi_1(x)+\varphi_2(y)$. Clearly, $\varphi'$ is an extension of $\varphi:M\to \mathbb C$. We claim that it is continous with respect to the metric $d_\Phi$. Suppose $(x_n,y_n)$ is a sequence in $X_1\oplus X_2$ converging to $(x_1,x_2)$ with respect to the metric $d_\Phi$. Then $$\lim_{n\to \infty}\Phi^{-1}\left(\Phi(\|x_n-x\|)+\Phi(\|y_n-y\|)\right)=0.$$
	 	Since $\Phi^{-1}$ and $\Phi$ are strictly increasing continuous functions, we can deduce that $\|x_n-x\|\to 0$ and $\|y_n-y\|\to 0$.  Hence, $x_n\to x$ and $y_n\to y$. Using the continuity of $\varphi_1$ and $\varphi_2$, we know that $\varphi_1(x_n)\to\varphi_1(x)$ and $\varphi_2(y_n)\to\varphi(y)$. Hence, $\varphi'(x_n,y_n)=\varphi_1(x_n)+\varphi_2(y_n)\to\varphi_1(x)+\varphi_2(y)=\varphi'(x,y)$. Thus, $\varphi'$ is an continuous extension of $\varphi$.
	 \end{proof}
	 
	 \begin{theorem}\label{locb}
	 	For a family of locally bounded $F$-spaces $\{X_i\}_{i=1}^n$ and an Young function $\Phi$ which satisfies Mulholland condition, the $F$-space $\left(\oplus_nX_i,d_\Phi\right)$ is locally bounded.
	 \end{theorem}
	 \begin{proof}
	 	It would be sufficient to prove this for the case $n=2$. Let $B_r^{(i)}$ denote the open ball of the F-space $X_i$ of radius $r$ and centered at $0$.  The open ball of $\left(X\oplus X_2,d_\Phi\right)$ centered at $(0,0)$ and of radius $r$ will be denoted by $B_r^\Phi$. Notice that for $(x,y)\in B_1$, we have 
	 	$\Phi^{-1}\left(\Phi(\|x\|)+\Phi(\|y\|)\right)< 1$. Thus $\Phi(\|x\|)<\Phi(1)$ and $\Phi(\|y\|)<1$. Hence, $\|x\|<1$ and similarly $\|y\|<1$. Thus, $(x,y)\in B_1^{(1)}\times B_1^{(2)}$ and 
	 	\begin{equation}\label{inc}
	 		B_1^\Phi\subset B_1^{(1)}\times B_1^{(2)}.
	 	\end{equation}
	 	Further, if $c,r>0$ then for any $(cx,cy)\in cB_{r}^{(1)}\times cB_{r}^{(2)}$, we have 
	 	\begin{align}
	 		\left|\left|\left(\frac{cx}{2c}.\frac{cy}{2c}\right)\right|\right|_\Phi&=\Phi^{-1}\left(\Phi(\frac{\|cx\|}{2c})+\Phi(\frac{\|cy\|}{2c})\right)\nonumber\\&\leq\Phi^{-1}\left(\frac{1}{2}\Phi(\|x\|)+\frac{1}{2}\Phi(\|y\|)\right)&&\text{due to convexity of}~\Phi\nonumber\\&\leq \Phi^{-1}\left(\frac{1}{2}\Phi(r)+\frac{1}{2}\Phi(r)\right)&&\text{due to}~\Phi~\text{being increasing}\nonumber\\&=r.
	 	\end{align}
	 	Thus, \begin{equation}\label{inc1}
cB_{r}^{(1)}\times cB_{r}^{2}\subset 2cB_r^{\Phi}~~~~~~~~~~~~~~~~~~~~~~~~~~~~~~~~~~~~~\text{for each}~c,r>0
	 	\end{equation}
	 	Now fix a $r_o<1$. Since $X_1$ and $X_2$ are locally, we can find a $t$ such that $B_1^{(1)}\subset sB_{r_o}^{(1)}$ and $B_1^{(2)}\subset sB_{r_o}^{(2)}$ for all $s>t$. Hence, 
	 	\begin{equation}
	 		B_1^{(1)}\times B_1^{(2)}\subset sB_{r_o}^{(1)}\times sB_{r_o}^{(2)}~~~~~~~~~~~~~~\forall s>t.
	 	\end{equation}
	 	
	 	Combining the above inclusion with the inclusion in equation-\ref{inc} and \ref{inc1}, we get 
	 	\begin{equation}
B_1^\Phi\subset	B_1^{(1)}\times B_1^{(2)}\subset sB_{r_o}^{(1)}\times sB_{r_o}^{(2)}\subset 2sB_{r_o}^\Phi~~\forall s>t.
	 	\end{equation}
	 	Thus, the collection of open balls of radius less than one and centered at origin forms a neighborhood base at origin. The set $B_1$ is bounded and hence $\left(X_1\oplus X_2,d_\Phi\right)$ is a locally bounded $F$-space.
	 \end{proof}
	\subsection{Coefficient of non-homogeneity of $F$-spaces}
	Since $F$-norms are non-homogeneous, the best property an $F$-space $(X,\|.\|_F)$ they could exhibit is that there exists $M>0$ such that $\|kx\|_F\leq M|k|\cdot \|X\|_F$  holds for all $x\in X$ and $k\in \mathbb C$. We define the coefficient of non homogeneity for $F$-spaces in the obvious way as follows.
	\begin{definition}
		Let $(X,\|.\|_F)$ be an $F$-space. Then the coefficient of non-homogeneity of $X$ is $$\nu_{(X,\|.\|_F)}=\inf\{M>0~:~\|kx\|\leq M|k|.\|x\|,~\forall k\in \mathbb C, x\in X\}.$$
	\end{definition}
	Obviously not all spaces have a finite coefficient of non-homogeneity. For example, if $0<p<1$, then $(\mathbb R^2,\|.\|_p)$ does not have finite coefficient of non-homogeneity because $\|(kx,ky)\|_p=k^p\|(x,y)\|_p$. We will show that $F$-norms associated to the Young functions with Mulholland condition, whose characteristic function has linear asymptote, have finite coefficient of non-homogeneity.\\
	
	\begin{lemma}\label{coeff}
		If $\Phi$ is a Young function with Mulholland condition such that its characteristic function has a line as its asymptote and $X_1$, $X_2$ are $F$-spaces with finite coefficient of non-homogeneity, then the coefficient of non-homogeneity of the $F$-space $(X_1\oplus X_2,\|.\|_\Phi)$ is $\max\{\nu_{X_1},\nu_{X_2}\}$. i.e, $\nu_{(X_1\oplus X_2,\|.\|_\Phi)}=\max\{\nu_{X_1},\nu_{X_2}\}$.
	\end{lemma}
	\begin{proof}
		Let $\chi(x)$ be the characteristic function of $\Phi$ and $\chi'(x)=mx-c$ be the asymptote to $\chi$. Further, let $\Phi'$ be the Young function whose characteristic function is $\chi'$. Then, $mx\leq \chi(x)\leq \chi'(x)$. And hence,
		$$\Phi''(x)=x^{m+1}\leq \Phi(x)\leq \frac{1}{e^c}x^{m+1}=\Phi'(x).$$
		Let $m_1=\nu_{X_1}$ and $m_2=\nu_{X_2}$ and $M=\max\{m_1,m_2\}$. Now notice that
	\begin{align}
			||k(x,y)||_\Phi&=\Phi^{-1}\left(\Phi(||kx||)+\Phi(||ky||)\right)\nonumber\\&\leq {\Phi''}^{-1}\left(\Phi'(m_1|k|||x||)+\Phi'(m_2|k|||y||)\right)\nonumber\\
			&\leq  \Phi''^{-1}\left(\frac{{M|k|}^{m+1}}{e^c}(\|x\|^{m+1}+\|y\|^{m+1})\right)\nonumber\\&=\frac{M|k|}{e^{\frac{c}{m+1}}}\left(\|x\|^{m+1}+\|y\|^{m+1}\right)^{\frac{1}{m+1}}\nonumber\\&=\frac{M|k|}{e^{\frac{c}{m+1}}}\Phi'^{-1}\left(\Phi'(\|x\|)+\Phi'(\|y\|)\right)\nonumber\\&\leq \frac{M|k|}{e^{\frac{c}{m+1}}}\Phi'^{-1}\left(e^c(\Phi(\|x\|)+\Phi(\|y\|))\right)\nonumber\\&=\frac{M|k|}{e^{\frac{c}{m+1}}}\cdot e^{\frac{c}{m+1}}\Phi'^{-1}\left(\Phi(\|x\|)+\Phi(\|y\|)\right)\nonumber\\&\leq M|k|\cdot\Phi^{-1}\left(\Phi(\|x\|)+\Phi(\|y\|)\right)
	\end{align}
	Thus, $\|(kx,ky)\|_\Phi\leq M|k|\cdot\|(x,y)\|_\Phi$ holds for all $k\in \mathbb C$ and $x\in X_1$, $y\in Y_1$. Hence, $\nu_{(X_1\oplus X_2,\|.\|_\Phi)}=\max\{\nu_{X_1},\nu_{X_2}\}$.
		\end{proof}
	
\section{Direct Sums of Orlicz spaces}
This section is dedicated to establishing a proper notion of direct sums of Orlicz spaces. Recall that if $(S,\mathcal A, \mu)$ is a sigma-finite measure space, then we can define the $p$-direct sum $L^p(S)\oplus_p L^p(S)$ equipped with the $p$-norm $||(f,g)||_p=\left(||f||^p+||g||^p\right)^{1/p}.$ Further the $L^p(S)\oplus_p L^p(S)$ is isometrically isomorphic to $L^p\left(\{1,2\}, L^p(S)\right)$ (Bochner space, see \cite[1.2b]{Hytonen}) through the identification
$\theta(f,g)(1)=f,~\theta(f,g)(2)=g$. Further $L^p\left(\{1,2\}, L^p(S)\right)$ is isometrically isomorphic to $L^p(\{1,2\}\times S)$ (see \cite[Prop. 1.2.24]{Hytonen}). Hence, the $p$-direct sum of $L^p(S)$ is again an $L^p$ space, albeit over a different measure space $\{1,2\}\times S$ (product measure of counting measure space $\{1,2\}$ and measure space $S$.)\\
Now suppose $\Phi$ is any Young function and $(S,\mathcal A,\mu)$ be a measure space. Consider the Orlicz space $L^\Phi(S)$ as defined in section \ref{Orlicz}. We aim to define an appropriate norm $N_\Gamma$ on the vector space $L^\Phi(S)\oplus L^\Phi(S)$ such that it becomes a $L^\Phi$ space on some measure space.\\

Let $\Gamma:\mathbb R^2\to [0,\infty)$ be any convex continuous function which is radially increasing (i.e, $\Gamma(rx,ry)$ is an increasing function of $r$ for a fixed $(x,y)$) and the contours $U_c=\{(x,y)~:~\Gamma(x,y)=c\}$ for any $c>0$ are all convex polygons with fixed number of sides and centered at origin. Further the extreme points of the polygon $U_1$ are no farther than one unit from origin and $\Gamma(0,0)=0$. Define $N_{\Gamma}$ on $L^\Phi(S)\oplus L^\Phi(S)$ as 
$$N_\Gamma(f,g)=\inf\left\{\lambda>0~:~\Gamma\left(\int_S\Phi\left(\frac{|f(t)|}{\lambda}\right)dt, \int_S\Phi\left(\frac{|g(t)|}{\lambda}\right)dt\right)\leq 1\right\}$$
We need to verify that this is a well defined norm. Clearly $N_\Gamma(f,g)\geq 0 $. If $f=h=0$ then $N_\Gamma(f,g)=0$ follows easily. Further if $N_\Phi(f,g)=0$, then there exists a sequence $\{\lambda_n\}_{n=1}^\infty \to 0$ such that $\Gamma\left(\int_S\Phi\left(\frac{|f(t)|}{\lambda_n}\right)dt, \int_S\Phi\left(\frac{|g(t)|}{\lambda_n}\right)dt\right)\leq 1$ for each $n$. Hence, $$\int_S\Phi\left(\frac{|f(t)|}{\lambda_n}\right)dt,\int_S\Phi\left(\frac{|g(t)|}{\lambda_n}\right)dt\in Int(U_1)$$
In particular $\int_S\Phi\left(\frac{|f(t)|}{\lambda_n}\right)dt\leq 1$ for each $n$. Now suppose that $f$ is a non-zero function, then there exists a set $A$ of positive measure and an $\epsilon>0$ such that $|f(t)|\geq \epsilon$ for all $t\in A$. Thus,
\begin{align*}
	\mu(A)\Phi\left(\frac{\epsilon}{\lambda_n}\right)&\leq \int_X\Phi\left(\frac{|f(t)|}{\lambda_n}\right)dt\\&\leq 1&&\text{for each}~n
	\end{align*}
Which means $\Phi(\frac{\epsilon}{\lambda_n})\leq \frac{1}{\mu(A)}$ for each $n$. But this is absurd because $\frac{\epsilon}{\lambda_n}$ increases indefinitely and $\Phi$ is an increasing function. Hence, contrary to our assumption, $f$ must be a zero almost everywhere function. Similarly, $g$ also vanishes almost everywhere. Now we verify the triangle inequality. Suppose $(f_1,g_1), (f_2,g_2)\in L^\Phi(S)\oplus L^\Phi(S)$ and $N_\Gamma(f_1,g_1)=k_1$ and $N_\Gamma(f_2,g_2)=k_2$. Now notice that 
\begin{align*}
&\Gamma\left(\int_S\Phi\left(\frac{|f_1(t)+f_2(t)|}{k_1+k_2}\right)dt,\int_S\Phi\left(\frac{|g_1(t)+g_2(t)|}{k_1+k_2}\right)dt\right)\\&\leq \Gamma\left(\frac{k_1}{k_1+k_2}\left(\int_S\Phi\left(\frac{|f_1(t)|}{k_1}\right)dt,\int_S\Phi\left(\frac{|g_1(t)|}{k_1}\right)dt\right)+\frac{k_2}{k_1+k_2}\left(\int_S\Phi\left(\frac{|f_1(t)|}{k_1}\right)dt,\int_S\Phi\left(\frac{|g_1(t)|}{k_1}\right)dt\right)\right)\\&\leq \frac{k_1}{k_1+k_2}\Gamma\left(\int_S\Phi\left(\frac{|f_1(t)|}{k_1}\right)dt,\int_S\Phi\left(\frac{|g_1(t)|}{k_1}\right)dt\right)+\leq \frac{k_2}{k_1+k_2}\Gamma\left(\int_S\Phi\left(\frac{|f_2(t)|}{k_2}\right)dt,\int_S\Phi\left(\frac{|g_2(t)|}{k_2}\right)dt\right)\\&\leq \frac{k_1}{k_1+k_2}+\frac{k_1}{k_1+k_2}\\&=1
\end{align*}
Hence, $N_\Gamma((f_1,g_1)+(f_2,g_2))\leq N_\Gamma(f_1,g_1)+N_\Gamma(f_2,g_2)$.

Finally, we verify that $L^\Phi(S)\oplus L^\Phi(S)$ is complete with respect to the norm $N_\Gamma$. Let $\{(f_i,g_i)\}_{i=1}^\infty$ be a Cauchy sequence in $L^\Phi(S)\oplus L^\Phi(S)$ with respect to  norm $N_\Gamma$. Then for each $\epsilon>0$, there exists $N_\epsilon\in \mathbb N$ such that $N_\Gamma((f_n,g_n)-(f_m,g_m))<\epsilon$ for all $n,m\geq N_\epsilon$. Hence for each pair of positive integers $(n,m)$ such that $n,m\geq N_\epsilon$, we can choose a $0<\lambda_{n,m}<\epsilon$ such that $$\Gamma\left(\int_X\Phi\left(\frac{|f_n(t)-f_m(t)|}{\lambda_{n,m}}\right)dt, \int_X\Phi\left(\frac{|g_n(t)-g_m(t)|}{\lambda_{n,m}}\right)\right)\leq 1~~~~\forall n,m\geq N_\epsilon.$$

Hence, $\int_X\Phi\left(\frac{|f_n(t)-f_m(t)|}{\lambda_{n,m}}\right)dt\leq 1$ for all $n,m\geq N_\epsilon$, which in turn means $N_\Phi(f_n-f_m)\leq\lambda_{n,m}\leq \epsilon$ for each $n,m\geq N_\epsilon$ Hence $\{f_n\}_{n=1}^\infty$ and similarly $\{g_n\}_{n=1}^\infty$ are Cauchy sequences in $L^\Phi(S)$. Suppose $f_n\to f\in L^\Phi(S)$ and $g_n\to g\in L^\Phi(S)$. Then with easy computations similar to above, one can conclude that $(f_n,g_n)$ converges to $(f,g)$ with respect to $N_\Gamma$ norm. Hence $(L^\Phi(S))\oplus L^\Phi(S)$ is complete with respect to $N_\Gamma$.\\

We now establish that the appropriate norm for the direct sum $L^\Phi(S)\oplus L^\Phi(S)$ corresponds to the function $\Gamma(x,y)=|x|+|y|$.

\begin{theorem}
	Let $\Gamma:\mathbb R^2\to [0,\infty)$ be the function $\Gamma(x,y)=|x|+|y|$ and $(S,\mathcal A,\mu)$ be a measure space. Then the map $\eta:(L^\Phi(S)\oplus L^\Phi(S),N_\Gamma)\to L^\Phi(\{1,2\}\times S)$ defined as $$\eta(f,g)(1,x)=f(x),~\eta(f,g)(2,x)=g(x)$$ is an isometric isomorphism of Banach spaces.
\end{theorem}
\begin{proof}
	Clearly $\eta$ is a linear bijection, as can be verified easily. Further for any $(f,g)\in L^\Phi(S)\oplus L^\Phi(S)$, we have
	\begin{align}\label{norm}
		N_\Phi(\eta(f,g))&=\inf\left\{\lambda>0~:~ \int_{\{1,2\}\times S}\Phi\left(\frac{|\eta(f,g)(t,s)|}{\lambda}\right)d\mu(s,t)\leq 1\right\}.
	\end{align}
	But \begin{align*}
\int_{\{1,2\}\times S}\Phi\left(\frac{|\eta(f,g)(t,s)|}{\lambda}\right)d\mu(s,t)&=\int_{ S}\Phi\left(\frac{|f(x)|}{\lambda}\right)dx+\int_{ S}\Phi\left(\frac{|g(x)|}{\lambda}\right)dx\\&=\Gamma\left(\int_{ S}\Phi\left(\frac{|f(x)|}{\lambda}\right)dx, \int_{ S}\Phi\left(\frac{|g(x)|}{\lambda}\right)dx\right)
	\end{align*}
	Hence, by equation \ref{norm}, we have  
	\begin{align*}
			N_\Phi(\eta(f,g))&=\inf\left\{\lambda>0~:~ \Gamma\left(\int_{ S}\Phi\left(\frac{|f(x)|}{\lambda}\right)dx, \int_{ S}\Phi\left(\frac{|g(x)|}{\lambda}\right)dx\right)\leq 1\right\}\\&=N_\Gamma(f,g)
	\end{align*}
	Thus, the isometric isomorphism is established.
\end{proof}
~\\
{\bf Data availability:} There is no data associated to this article.\\
~\\{\bf Statements and Declarations:}
Both the authors have equal contributions in this article. This research was supported by the Institute of Mathematics, Physics and Mechanics, Ljubljana (Slovenia), within the research program \say{Algebra,  Operator theory and Financial mathematics (code: P1-0222)}. Authors declare that there are no other competing interests that could have influenced the work reported in this article.\\

\end{document}